\documentclass[a4paper,12pt]{article}

\makeatletter
\@addtoreset{footnote}{page}
\makeatother

\usepackage{amsthm}
\usepackage{amsmath,amssymb,latexsym,amsfonts,mathrsfs,xcolor}

\theoremstyle{plain}
\newtheorem{Thm}{Theorem}[section]

\newtheorem{Lem}[Thm]{Lemma}

\newtheorem{Cor}[Thm]{Corollary}

\theoremstyle{definition}

\numberwithin{equation}{section}

\makeatletter
\renewcommand\section{\@startsection {section}{1}{\z@}%
                                   {-3.5ex \@plus -1ex \@minus -.2ex}%
                                   {2.3ex \@plus.2ex}%
                                   {\normalfont\large\bf}}
\makeatother

\makeatletter
\renewcommand\subsection{\@startsection {subsection}{1}{\z@}%
                                   {-3.5ex \@plus -1ex \@minus -.2ex}%
                                   {2.3ex \@plus.2ex}%
                                   {\normalfont\normalsize\bf}}
\makeatother

\allowdisplaybreaks

\begin{document}

\begin{center}
{\Large \bf 
Functional central limit theorems for the dynamic elephant random walk
}
\end{center}
\begin{center}
Go Tokumitsu
\end{center}

\begin{abstract}
We prove functional central limit theorems for the dynamic elephant random walk in the $\sqrt{n}$ and $\sqrt{n\log n}$ orders, by applying the martingale convergence theorem and Karamata's theory of regular variation. 
\end{abstract}

%%%%% text %%%%%

\section{Introduction}
The Elephant Random Walk (ERW), introduced by Sch\"utz and Trimper \cite{Sch} in 2004, is a one-dimensional discrete-time process with infinite memory that provides a simple example of anomalous diffusion and exhibits a phase transition; see \cite{BaurBertoin}, \cite{Ber}, \cite{Col2}, \cite{KaT} for the details. The Dynamic Random Walk (DRW), proposed by Guillotin \cite{Gui} in 2000, is a one-dimensional discrete-time random walk whose transition probabilities are governed by the orbit of a discrete-time dynamical system; we refer to \cite{Gui2} for more details. Coletti, de Lima, Gava, and Luiz \cite{Col} introduced the Dynamic Elephant Random Walk (DERW) in 2020, which randomly combines the ERW and DRW models. Their work revealed that the DERW exhibits another phase transition in addition to the ERW phase transition.

In this paper, we prove functional central limit theorems for the DERW in the $\sqrt{n}$ and $\sqrt{n\log n}$ orders. For this purpose we investigate the asymptotic behavior of characteristic sequences, following the approach of Shiozawa \cite{Sio}. Then, we apply the martingale convergence theorem of Durrett--Resnick \cite{DurrettResnick} to establish the limit theorems.
\subsection{The definition of the dynamic elephant random walk}
We now introduce an equivalent formulation of the DERW, as given in Tokumitsu--Yano \cite{TokumitsuYano2025Pre}. Let $p,q\in[0,1]$ be constants and let $\alpha=\{\alpha_n\}_{n=1}^\infty$ and $\beta=\{\beta_n\}_{n=1}^\infty$ be sequences of $[0,1]$. The DERW is a stochastic process $\{S_n\}_{n=0}^\infty$ on a probability space $(\Omega,\mathcal{F},P)$ which takes values in $\mathbb{Z}$ such that $S_0:=0$ and the increment $X_n:=S_n-S_{n-1}$ for $n\ge1$ takes values in $\{-1,+1\}$ and satisfies
\begin{align}\notag
E[X_1]=\alpha_1(2q-1)+(1-\alpha_1)(2\beta_1-1)
\end{align}
and
\begin{align}\notag
E[X_{n+1}|\mathcal{F}_n^X]=\frac{\alpha_{n+1}(2p-1)}{n}\cdot S_n+(1-\alpha_{n+1})(2\beta_{n+1}-1)\quad (n\ge1).
\end{align}
Here $\mathcal{F}_n^X:=\sigma(X_1,\ldots,X_n)$ denotes the natural filtration generated by $\{X_n\}_{n=1}^\infty$.

The DERW generalizes the ERW, as well as the models studied by Shiozawa \cite{Sio} and Kubota--Takei \cite{KaT}. For the details, see \cite{TokumitsuYano2025Pre}. 

Let us state our main results. Let $D([0,\infty))$ denote the Skorokhod space, i.e., the set of right-continuous functions with left-hand limits equipped with the Skorokhod topology. The following functional central limit theorem in the $\sqrt{n}$ order.
\begin{Thm}\label{mThm1}
Let $\{S_n\}_{n=0}^\infty$ be the DERW. Assume that $\alpha_n\to\alpha$ and $\beta_n\to\beta$, and suppose one of the following five conditions:
\begin{enumerate}
\item $p< 3/4$, $\alpha>0$.
\item $p=3/4$, $0<\alpha<1$.
\item $3/4<p<1$, $0<\alpha<\frac{1}{4p-2}$.
\item $p=1$, $0<\alpha<\frac{1}{4p-2}$, $0<\beta<1$.
\item $\alpha=0$, $0<\beta<1$.
\end{enumerate}
Then we have the following distributional convergence in $D([0,\infty)):$
\begin{align}\label{FCLT}
\left(\frac{S_{[nt]}-E[S_{[nt]}]}{\sqrt{n}},t\ge0\right)\Longrightarrow(W_t,t\ge0),
\end{align}
where $(W_t,t\ge0)$ is a $\mathbb{R}$-valued centered Gaussian process starting at the origin with covariance given by 
\begin{align}\notag
E[W_sW_t]=C_{\alpha,\beta,p}\cdot s\left(\frac{t}{s}\right)^{\alpha(2p-1)},
\end{align}
for all $0<s\le t$. In particular,
\begin{align}\notag
\frac{S_n-E[S_n]}{\sqrt{n}}\overset{d}{\to}N\left(0,C_{\alpha,\beta,p}\right).
\end{align}
Here,
\begin{align}\label{Const}
C_{\alpha,\beta,p}:=\frac{1}{1-2\alpha(2p-1)}\cdot\left(1-\left(\frac{(1-\alpha)(2\beta-1)}{(1-\alpha(2p-1)}\right)^2\right).
\end{align}
\end{Thm}
Condition (i) is called the \emph{diffusive} case, conditions (iii) and (iv) are called the \emph{superdiffusive} case, condition (v) is called the \emph{asymptotically no elephant component} case. Condition (ii) corresponds to the critical case.

For another critical case [$3/4\le p\le1$ and $\alpha_n\equiv\frac{1}{4p-2}$], we obtain the following functional central limit theorem in the $\sqrt{n\log n}$ order.
\begin{Thm}\label{mThm2}
Let $\{S_n\}_{n=0}^\infty$ be the DERW. Assume that $\alpha_n\equiv \frac{1}{4p-2}$ and $\beta_n\to\beta$, and suppose one of the following two conditions:
 \begin{enumerate}
\item $3/4\le p<1$.
\item $p=1$, $0<\beta<1$.
\end{enumerate}
Then we have the following distributional convergence in $D([0,\infty)):$
\begin{align}\label{FCLT2}
\left(\frac{S_{[n^t]}-E[S_{[n^t]}]}{\sqrt{n^t\log n}},t\ge0\right)\Longrightarrow\left(\sqrt{C_{\beta,p}'}B_t,t\ge0\right),
\end{align}
where $(B_t,t\ge0)$ is a standard Brownian motion. In particular,
\begin{align}
\frac{S_n-E[S_n]}{\sqrt{n\log n}}\overset{d}{\to}N\left(0,C_{\beta,p}'\right).
\end{align}
Here,
\begin{align}\notag
C_{\beta,p}':=1-4\left(1-\frac{1}{4p-2}\right)^2(2\beta-1)^2.
\end{align}
\end{Thm}
\subsection{Previous results}
The following theorem states the functional central limit theorem for the ERW in the $\sqrt{n}$ order.
\begin{Thm}[Baur--Bertoin {\cite[Theorem 1]{BaurBertoin}}]\label{FCLTERW1}
Let $\{S_n\}_{n=0}^\infty$ be the ERW with $p<3/4$. Then we have the following distributional convergence. in $D([0,\infty)):$
\begin{align}\notag
\left(\frac{S_{[nt]}}{\sqrt{n}},t\ge0\right)\Longrightarrow(W_t,t\ge0),
\end{align}
where $(W_t,t\ge0)$ is a $\mathbb{R}$-valued centered Gaussian process starting at the origin with covariance given by 
\begin{align}\notag
E[W_sW_t]=\frac{s}{3-4p}\left(\frac{t}{s}\right)^{2p-1},
\end{align}
for all $0<s\le t$.
\end{Thm}
The following theorem states the functional central limit theorem for the ERW in the $\sqrt{n\log n}$ order.
\begin{Thm}[Baur--Bertoin {\cite[Theorem 2]{BaurBertoin}}]\label{FCLTERW2}
Let $\{S_n\}_{n=0}^\infty$ be the ERW with $p=3/4$. Then we have the following distributional convergence. in $D([0,\infty)):$
\begin{align}\notag
\left(\frac{S_{[n^t]}}{\sqrt{n^t\log n}},t\ge0\right)\Longrightarrow(B_t,t\ge0).
\end{align}
\end{Thm}
Theorem \ref{FCLTERW1} is a special case of Theorem \ref{mThm1}, and Theorem \ref{FCLTERW2} is a special case of Theorem \ref{mThm2}.

We will write for short 
\begin{align*}
\ell_{\inf}(\alpha):=\liminf_{n\to\infty}\alpha_n,\quad \ell_{\sup}(\alpha):=\limsup_{n\to\infty}\alpha_n,
\end{align*}
and
\begin{align*}
\ell_{\inf}(\beta):=\liminf_{n\to\infty}\beta_n,\quad \ell_{\sup}(\beta):=\limsup_{n\to\infty}\beta_n.
\end{align*}
Let $a_1:=1$ and
\begin{align*}
a_n:=\prod_{k=1}^{n-1}\left(1+\frac{(2p-1)\alpha_{k+1}}{k}\right)\quad \text{for}\ n\ge2.
\end{align*}

Set
\begin{align*}
A_n^2:=\sum_{k=1}^n\frac{1}{a_k^2}(1-E[X_k]^2),\quad B_n^2:=\sum_{k=1}^n\frac{1}{a_k^2}\quad \text{for $n\in\mathbb{N}\cup\{\infty\}$}.
\end{align*}
Coletti et al. \cite{Col} have obtained the following theorem concerning the Central Limit Theorem (CLT).
\begin{Thm}[Coletti et al. {\cite[Theorem 3, and Corollary 4]{Col}}]
Let $\{S_n\}_{n=0}^\infty$ be the DERW. Suppose one of the following four conditions:
\begin{enumerate}
\item $p\le 3/4$ and $\ell_{\inf}(\alpha)>0$.
\item $3/4<p<1$, $0<\ell_{\inf}(\alpha)\le\ell_{\sup}(\alpha)<\frac{1}{4p-2}$.
\item $p=1$ and $0<\ell_{\inf}(\alpha)\le\ell_{\sup}(\alpha)<1$, $0<\ell_{\inf}(\beta)\le\ell_{\sup}(\beta)<\frac{1-\ell_{\sup}(\alpha)}{1-\ell_{\inf}(\alpha)}$.
\item $0=\ell_{\inf}(\alpha)\le \ell_{\sup}(\alpha)<1$ and $0<\ell_{\inf}(\beta)\le\ell_{\sup}(\beta)<1-p\cdot\ell_{\sup}(\alpha)$.
\end{enumerate}
Then the following assertion hold:
\begin{align}\label{CLT}
\frac{S_n-E[S_n]}{a_nA_n}\overset{d}{\to}N(0,1).
\end{align}
\end{Thm}
In addition to the phase transition at $p=3/4$, Coletti et al. \cite{Col} obtained another phase transition as the following theorem with a \emph{strong elephant component} in the sense that $\ell_{\inf}(\alpha)>\frac{1}{4p-2}$, where the CLT \eqref{CLT} breaks down.
\begin{Thm}[Coletti et al. {\cite[Theorem 7]{Col}}]\label{nondegThm}
Let $\{S_n\}_{n=0}^\infty$ be the DERW with $p>3/4$ and $\ell_{\inf}(\alpha)>\frac{1}{4p-2}$. Then
\begin{align}\label{nondeg}
\frac{S_n-E[S_n]}{a_n}\to M\quad \text{a.s. and in $L^2$},
\end{align} 
where $M$ is a non-degenerate zero mean random variable.
\end{Thm}
The following theorem states the fluctuation limit of \eqref{nondeg} of Theorem \ref{nondegThm}.
\begin{Thm}[Tokumitsu--Yano {\cite[Theorem 1.3]{TokumitsuYano2025Pre}}]
\label{TY2025}
Let $\{S_n\}_{n=0}^\infty$ be the DERW. Suppose one of the following two conditions:
\begin{enumerate}
\item $3/4<p<1$ and $\ell_{\inf}(\alpha)>\frac{1}{4p-2}.$
\item $p=1$ and $\frac{1}{4p-2}<\ell_{\inf}(\alpha)\le\ell_{\sup}(\alpha)<1$, $0<\ell_{\inf}(\beta)\le\ell_{\sup}(\beta)<\frac{1-\ell_{\sup}(\alpha)}{1-\ell_{\inf}(\alpha)}$.
\end{enumerate} 
Then the following assertion hold:
\begin{align*}
\frac{S_n-E[S_n]-a_nM}{a_n\sqrt{A_\infty^2-A_n^2}}\overset{d}{\to}N(0,1).
\end{align*}
\end{Thm}
\subsection{Organization of this paper}
This paper is organized as follows. In Section 2, we study limit theorems for the DERW and the asymptotic behavior of characteristic sequences. In Section 3, we give the proof of Theorems \ref{mThm1} and \ref{mThm2}.
\section{Limit theorems for the DERW and the asymptotic behavior of characteristic sequences}
A function $L:(0,\infty)\to(0,\infty)$ is said to be \emph{slowly varying} if 
\begin{align}\notag
\frac{L(tx)}{L(x)}\xrightarrow[x\to\infty]{}1,
\end{align}
for any $t>0$; see \cite{Bingham} for the details. A sequence of positive numbers $\{a_n\}_{n=1}^\infty$ is called \emph{slowly varying} if there exists a slowly varying $L:(0,\infty)\to(0,\infty)$ such that $a_n = L(n)$ for all $n\ge1$. 

Set $g_1=\ell_1=\rho_1=\rho_2=1$ and
\begin{align*}
g_n&:=\prod_{k=1}^{n-1}\left(1+\frac{\alpha(2p-1)}{k}\right)\quad n\ge2,\\
\ell_n&:=\prod_{k=1}^{n-1}\left(1+\frac{(2p-1)(\alpha_{k+1}-\alpha)}{k+(2p-1)\alpha}\right)\quad n\ge2,\\
\rho_n&:=\exp\left(\sum_{k=2}^{n-1}\frac{(2p-1)(\alpha_{k+1}-\alpha)}{k}\right)\quad n\ge3.
\end{align*}
By definition, we have $a_n=g_n\ell_n$ for all $n\ge1$. By Stirling's formula, we have
\begin{align}\label{gasym}
g_n\sim\frac{n^{\alpha(2p-1)}}{\Gamma(\alpha(2p-1)+1)},
\end{align}
where $\Gamma$ denotes the Gamma function.
\begin{Lem}\label{Lem1}
Assume that $\alpha_n\to\alpha\in[0,1]$ and $\beta_n\to\beta\in[0,1]$. Then the following assertions hold:
\begin{enumerate}
\item The sequences $\{\rho_n\}_{n=1}^\infty$ and $\{\ell_n\}_{n=1}^\infty$ are slowly varying.
\item $\{a_n\}_{n=1}^\infty$ is a regularly varying sequence with index $\alpha(2p-1)$ such that
\begin{align}\notag
a_n\sim \frac{n^{\alpha(2p-1)}}{\Gamma(\alpha(2p-1)+1)}\ell_n.
\end{align}
\end{enumerate}
\end{Lem}
\begin{proof}
(i) Since
\begin{align*}
\left|\sum_{k=[x]}^{[xt]-1}\frac{(2p-1)(\alpha_{k+1}-\alpha)}{k}\right|&\le\sup_{k\ge [x]}|\alpha_{k+1}-\alpha|\sum_{k=[x]}^{[xt]-1}\frac{1}{k}\\
&\le\sup_{k\ge [x]}|\alpha_{k+1}-\alpha|\left(\int_{[x]-1}^{[xt]}\frac{dx}{x}\right)\\
&\le\sup_{k\ge [x]}|\alpha_{k+1}-\alpha|\log\frac{[xt]}{[x]-1}\xrightarrow[x\to\infty]{}0,
\end{align*}
we have
\begin{align}\notag
\frac{\rho_{[xt]}}{\rho_{[x]}}=\exp\left(\sum_{k=[x]}^{[xt]-1}\frac{(2p-1)(\alpha_{k+1}-\alpha)}{k}\right)\xrightarrow[x\to\infty]{}1,
\end{align}
which shows that $\{\rho_n\}_{n=1}^\infty$ is a slowly varying sequence. By the Taylor theorem, for any $x>-1$, there exists $\theta_x\in(0,1)$ such that
\begin{align}\notag
\log(1+x)=x-\frac{(\theta_x x)^2}{2}.
\end{align} 
For any $k\ge2$, since $(2p-1)(\alpha_k-\alpha)/(k+(2p-1)\alpha)>-1$, there exists $\theta_k\in(0,1)$ such that
\begin{align*}
\log\left(1+\frac{(2p-1)(\alpha_{k+1}-\alpha)}{k+(2p-1)\alpha}\right)&=\frac{(2p-1)(\alpha_{k+1}-\alpha)}{k+(2p-1)\alpha}-\frac{\theta_k^2}{2}\left(\frac{(2p-1)(\alpha_{k+1}-\alpha)}{k+(2p-1)\alpha}\right)^2\\
&=\frac{(2p-1)(\alpha_{k+1}-\alpha)}{k}-\frac{(2p-1)^2\alpha(\alpha_{k+1}-\alpha)}{k(k+(2p-1)\alpha)}\notag\\
&\quad -\frac{\theta_k^2}{2}\left(\frac{(2p-1)(\alpha_{k+1}-\alpha)}{k+(2p-1)\alpha}\right)^2.
\end{align*}
Summing both sides from $k=2$ to $n-1$ and taking their exponentials, we have
\begin{align*}
\ell_n&\times\left(1 + \frac{(2p-1)(\alpha_1 - \alpha)}{1 + (2p-1)\alpha} \right)^{-1} \notag \\
&= \rho_n\times\exp\left(-\left(\sum_{k=2}^{n-1}\frac{(2p-1)^2\alpha(\alpha_{k+1}-\alpha)}{k(k+(2p-1)\alpha)}+\frac{\theta_k^2}{2} \left(\frac{(2p-1)(\alpha_{k+1} - \alpha)}{k + (2p-1)\alpha}\right)^2\right)\right).
\end{align*}
Set
\begin{align*}
C:=&\exp\left(-\left(\sum_{k=2}^{\infty}\frac{(2p-1)^2\alpha(\alpha_{k+1}-\alpha)}{k(k+(2p-1)\alpha)}+\frac{\theta_k^2}{2}\left(\frac{(2p-1)(\alpha_{k+1}-\alpha)}{k+(2p-1)\alpha}\right)^2\right)\right)\notag\\
&\times\left(1+\frac{(2p-1)(\alpha_1-\alpha)}{1+(2p-1)\alpha}\right),
\end{align*}
then $\ell_n\sim C\rho_n$. Since $\{\rho_n\}_{n=1}^\infty$ is a slowly varying sequence, we see that $\{\ell_n\}_{n=1}^\infty$ is also a slowly varying sequence. 

\noindent(ii) We have
\begin{align}\notag
a_n=g_n\ell_n\sim \frac{n^{\alpha(2p-1)}}{\Gamma(\alpha(2p-1)+1)}\ell_n
\end{align} 
by Eq. \eqref{gasym}. The proof is complete.
\end{proof}
\begin{Thm}\label{SLLN}
Let $\{S_n\}_{n=0}^\infty$ be the DERW. If $\alpha_n\to\alpha$ and $p\cdot\alpha<1$, then
\begin{align}\notag
\lim_{n\to\infty}\frac{S_n}{n}=\frac{(1-\alpha)(2\beta-1)}{1-(2p-1)\alpha},\quad \text{a.s.}
\end{align}
\end{Thm}
\begin{proof}
By Theorem 1 of \cite{Col}, we have
\begin{align}\notag
\lim_{n\to\infty}\frac{S_n-E[S_n]}{n}=0,\quad\text{a.s.}
\end{align}
Therefore, it suffices to show 
\begin{align}\notag
\frac{E[S_n]}{n}\to\frac{(1-\alpha)(2\beta-1)}{1-(2p-1)\alpha}\quad\text{as $n\to\infty$}.
\end{align} 
By Proposition 9 of \cite{Col} and Lemma 8 of \cite{Col}, we have
\begin{align}\notag
\frac{E[S_n]}{n}=o(1)+\frac{a_n}{n}\sum_{k=1}^n\frac{(1-\alpha_k)(2\beta_k-1)}{a_k}.
\end{align}
Since $\alpha(2p-1)<1$, we have $\sum_{n=1}^\infty a_n^{-1}=\infty$ by Lemma \ref{Lem1} (ii). By Stolz--Ces\`aro theorem, we have
\begin{align}\label{Stolz}
\lim_{n\to\infty}\frac{\sum_{k=1}^n(1-\alpha_k)(2\beta_k-1)a_k^{-1}}{\sum_{k=1}^na_k^{-1}}=(1-\alpha)(2\beta-1).
\end{align}
Since $-\alpha(2p-1)>-1$, by Proposition 1.5.8 of \cite{Bingham}, we have
\begin{align}\notag
\frac{a_n}{n}\sum_{k=1}^n\frac{1}{a_k}&\sim \frac{1}{n}\frac{n^{\alpha(2p-1)}\ell_{n}}{\Gamma(\alpha(2p-1)+1)}\frac{1}{1-\alpha(2p-1)}\frac{\Gamma(\alpha(2p-1)+1)}{n^{\alpha(2p-1)-1}\ell_n}\\
\label{Karamata}&\sim \frac{1}{1-\alpha(2p-1)}.
\end{align}
By Eq. \eqref{Stolz} and Eq. \eqref{Karamata}, we have
\begin{align*}
\frac{a_n}{n}\sum_{k=1}^n\frac{(1-\alpha_k)(2\beta_k-1)}{a_k}&=\frac{a_n}{n}\frac{\sum_{k=1}^n(1-\alpha_k)(2\beta_k-1)a_k^{-1}}{\sum_{k=1}^na_k^{-1}}\sum_{k=1}^n\frac{1}{a_k}\\
&\xrightarrow[n\to\infty]{}
\frac{(1-\alpha)(2\beta-1)}{1-(2p-1)\alpha}.
\end{align*}
The proof is complete.
\end{proof}
\begin{Lem}\label{Lem2}
Assume that $\alpha_n\to\alpha$ and $\beta_n\to\beta$. Then the following assertions hold:
\begin{enumerate}
\item Suppose one of the five conditions in Theorem \ref{mThm1}. Then
\begin{align}\notag
A_n^2\sim C_{\alpha,\beta,p}\Gamma(\alpha(2p-1)+1)^2n^{1-2\alpha(2p-1)}\ell_{n}^{-2}.
\end{align}
Consequently, $A_n^2$ is regularly varying with index $1-2\alpha(2p-1)$. Here $C_{\alpha,\beta,p}$ has been given in \eqref{Const}.
\item If $p>3/4$ and $\alpha>\frac{1}{4p-2}$, then
\begin{align}\notag
A_\infty^2-A_n^2\sim C_{\alpha,\beta,p}\Gamma(\alpha(2p-1)+1)^2n^{1-2\alpha(2p-1)}\ell_{n}^{-2}.
\end{align}
Consequently, $A_\infty^2-A_n^2$ is  regularly varying with index $1-2\alpha(2p-1)$. Here, 
\begin{align}\label{Const2}
C_{\alpha,\beta,p}:=\frac{1}{2\alpha(2p-1)-1}\cdot\left(1-\left(\frac{(1-\alpha)(2\beta-1)}{(1-\alpha(2p-1)}\right)^2\right).
\end{align}
\end{enumerate}
\end{Lem}
\begin{proof}
(i) By Theorem \ref{SLLN}, we have
\begin{align*}
\lim_{n\to\infty}(1-E[X_n]^2)&=1-\left(\frac{(2p-1)\alpha(1-\alpha)(2\beta-1)}{1-(2p-1)\alpha}+(1-\alpha)(2\beta-1)\right)^2\\
&=1-\frac{(1-\alpha)^2(2\beta-1)^2}{(1-(2p-1)\alpha)^2}.
\end{align*}
Since $-2\alpha(2p-1)>-1$, by Proposition 1.5.8 of \cite{Bingham}, we have
\begin{align*}
A_n^2&\sim \left(1-\frac{(1-\alpha)^2(2\beta-1)^2}{(1-(2p-1)\alpha)^2}\right)\frac{\Gamma(\alpha(2p-1)+1)^2}{1-2\alpha(2p-1)}n^{1-2\alpha(2p-1)}\ell_n^{-2}\\
&\sim C_{\alpha,\beta,p}\Gamma(\alpha(2p-1)+1)^2n^{1-2\alpha(2p-1)}\ell_{n}^{-2}.
\end{align*}
(ii) Since $-2\alpha(2p-1)<-1$, by Propositon 1.5.10 of \cite{Bingham}, we have
\begin{align}\notag
A_\infty^2-A_n^2\sim C_{\alpha,\beta,p}\Gamma(\alpha(2p-1)+1)^2n^{1-2\alpha(2p-1)}\ell_{n}^{-2}.
\end{align}
The proof is complete.
\end{proof}
\begin{Cor}
Let $\{S_n\}_{n=0}^\infty$ be the DERW. Assume that $\alpha_n\to\alpha$ and $\beta_n\to\beta$, and suppose one of the following two conditions:
\begin{enumerate}
\item $3/4<p<1$ and $\alpha>\frac{1}{4p-2}.$
\item $p=1$ and $\frac{1}{4p-2}<\alpha<1$, $0<\beta<1$.
\end{enumerate} 
Then the following assertion hold:
\begin{align*}
\frac{S_n-E[S_n]-a_nM}{\sqrt{n}}\overset{d}{\to}N(0,C_{\alpha,\beta,p}).
\end{align*}
Here, $C_{\alpha,\beta,p}$ is the constant defined in \eqref{Const2}.
\end{Cor}
\begin{proof}
By Lemma \ref{Lem1} (ii) and Lemma \ref{Lem2} (ii),
\begin{align}\notag
a_n\sqrt{A_\infty^2-A_n^2}\sim \sqrt{n}\cdot\sqrt{C_{\alpha,\beta,p}}.
\end{align}
Therefore, it follows from Theorem \ref{TY2025}.
\end{proof}
\section{Proof of our main theorems}
Set $M_0:=0$ and
\begin{align*}
M_n:=\frac{S_n-E[S_n]}{a_n}\quad \text{for}\ n\ge1,\\
Y_n:=M_{n}-M_{n-1}\quad \text{for} \ n\ge1.
\end{align*}
By Proposition 11 of \cite{Col}, we see that $\{M_n\}_{n=1}^\infty$ is a martingale. Since $Y_n=M_n-E[M_n|\mathcal{F}_{n-1}]\ \text{a.s.}$, we have
\begin{align*}
Y_n=\frac{X_n-E[X_n|\mathcal{F}_{n-1}]}{a_n}\quad \text{a.s.}
\end{align*}
for $n\ge1$. By the definition of the DERW, we have
\begin{align}\label{Eq}
|Y_n|\le\frac{2}{a_n}\quad\text{a.s.}
\end{align}
for $n\ge1$. Also, set
\begin{align}\notag
X_{n,k}:=\frac{Y_k}{A_n}
\end{align}
for $n,k\ge1$. Since $\{Y_k\}_{k=1}^\infty$ is martingale difference sequence, $\{X_{n,k}\}_{k=1}^\infty$ is also a martingale difference sequence for each fixed $n$.

\begin{proof}[Proof of Theorem \ref{mThm1}]
By Theorem 2.5 of \cite{DurrettResnick}, Theorem \ref{mThm1} can be proved by checking the following two conditions:
\begin{enumerate}
\item[(a)] For any $t>0$, $\displaystyle
\sum_{k=1}^{[nt]}E[X_{n,k}^2|\mathcal{F}_{k-1}^X]\xrightarrow[n\to\infty]{\mathrm{P}}\varphi(t)$ with $P(\text{$\varphi$ is continuous})=1$.
\item[(b)] For any $\varepsilon>0$, $\displaystyle
\sum_{k=1}^{n}E[X_{n,k}^21_{\{|X_{n,k}|>\varepsilon\}}|\mathcal{F}_{k-1}^X]\xrightarrow[n\to\infty]{\mathrm{P}}0$.
\end{enumerate}
Condition (b) has already been proven in the proof of Theorem 3 in \cite{Col}, so it suffices to prove only Condition (a). By Theorem 1 of \cite{Col}, 
\begin{align}\notag
E[X_{n}|\mathcal{F}_{n-1}^X]-E[X_n]=(2p-1)\alpha_{n}\frac{S_{n-1}-E[S_{n-1}]}{n-1}=o(1)\quad \text{a.s.}
\end{align}
Put
\begin{align}\notag
\xi_n:=E[X_{n}|\mathcal{F}_{n-1}^X]-E[X_n],\quad n\ge1.
\end{align}
Then we have
\begin{align*}
E[Y_n^2|\mathcal{F}_{n-1}^X]&=\frac{1}{a_n^2}E[1-2X_nE[X_n|\mathcal{F}_{n-1}^X]+E[X_n|\mathcal{F}_{n-1}^X]^2|\mathcal{F}_{n-1}^X]\\
&=\frac{1}{a_n^2}(1-E[X_n|\mathcal{F}_{n-1}^X]^2)\\
&=\frac{1}{a_n^2}(1-E[X_n]^2-2E[X_n]\xi_{n}-\xi_{n}^2).
\end{align*}
Thus, we have
\begin{align}\notag
\sum_{k=1}^{[nt]}E[X_{n,k}^2|\mathcal{F}_{k-1}^X]&=\frac{A_{[nt]}^2}{A_n^2}-\frac{1}{A_n^2}\left(\sum_{k=1}^{[nt]}\frac{2\xi_kE[X_k]+\xi_k^2}{a_k^2}\right). 
\end{align}
Let $\varepsilon>0$. By $2\xi_kE[X_k]+\xi_k^2\to0$ as $n\to\infty$, There exists $n_0$ such that for all $n\ge n_0$, $|2\xi_nE[X_n]+\xi_n^2|<\varepsilon$. Also, by Lemma \ref{Lem2} (i), $B_{[nt]}/A_{n}$ is bounded and $A_\infty=\infty$, we have
\begin{align*}
\left|\frac{1}{A_n^2}\left(\sum_{k=1}^{[nt]}\frac{2\xi_kE[X_k]+\xi_k^2}{a_k^2}\right)\right|&=\frac{1}{A_n^2}\left|\sum_{k=1}^{n_0-1}\frac{2\xi_kE[X_k]+\xi_k^2}{a_k^2}+\sum_{k=n_0}^{[nt]}\frac{2\xi_kE[X_k]+\xi_k^2}{a_k^2}\right|\\
&\le \frac{1}{A_n^2}\sum_{k=1}^{n_0-1}\frac{2\xi_kE[X_k]+\xi_k^2}{a_k^2}+\varepsilon\cdot\sup_{n\ge1}\frac{B_{[nt]}^2}{A_n^2}\\
&\xrightarrow[n\to\infty]{}\varepsilon \cdot \sup_{n\ge1}\frac{B_{[nt]}^2}{A_n^2}\xrightarrow[\varepsilon\downarrow0]{}0.
\end{align*}
Moreover, by Lemma \ref{Lem2} (i), we have
\begin{align}\notag
\frac{A_{[nt]}^2}{A_n^2}\sim\frac{[nt]^{1-2\alpha(2p-1)}}{n^{1-2\alpha(2p-1)}}\frac{\ell_{[nt]}^{-2}}{\ell_n^{-2}}\sim t^{1-2\alpha(2p-1)}:=\varphi(t).
\end{align}
Then Condition (a) is satisfied. Therefore, 
\begin{align}\notag
\frac{M_{[nt]}}{A_n}=\frac{S_{[nt]}-E[S_{[nt]}]}{a_{[nt]}A_n}\Longrightarrow B\circ\varphi(t),
\end{align}
where $B$ is a standard Brownian motion. By Lemma \ref{Lem1} (ii) and Lemma \ref{Lem2} (i),
\begin{align*}
a_{[nt]}A_n&\sim\frac{[nt]^{\alpha(2p-1)}}{\Gamma(\alpha(2p-1)+1)}\ell_{[nt]}\cdot \sqrt{C_{\alpha,\beta,p}}\Gamma(\alpha(2p-1)+1)n^{1/2-\alpha(2p-1)}\ell_{n}^{-1}\\
&\sim\sqrt{C_{\alpha,\beta,p}}t^{\alpha(2p-1)}\sqrt{n}.
\end{align*}
Therefore, we obtain the desired convergence \eqref{FCLT}, where $W_t:=\sqrt{C_{\alpha,\beta,p}}t^{\alpha(2p-1)}B\circ\varphi(t)$.
\end{proof}

\begin{proof}[Proof of Theorem \ref{mThm2}]
By Theorem 2.5 of \cite{DurrettResnick}, Theorem \ref{mThm2} can be proved by checking the following two conditions:
\begin{enumerate}
\item[(a$'$)] For any $t>0$, $\displaystyle
\sum_{k=1}^{[n^t]}E[X_{n,k}^2|\mathcal{F}_{k-1}^X]\xrightarrow[n\to\infty]{\mathrm{P}}\psi(t)$ with $P(\text{$\psi$ is continuous})=1$.
\item[(b)] For any $\varepsilon>0$, $\displaystyle
\sum_{k=1}^{n}E[X_{n,k}^21_{\{|X_{n,k}|>\varepsilon\}}|\mathcal{F}_{k-1}^X]\xrightarrow[n\to\infty]{\mathrm{P}}0$.
\end{enumerate}
When $\alpha_n\equiv\frac{1}{4p-2}$, since $a_n\sim 4n^{1/2}/\pi$, it follows from Theorem \ref{SLLN} that 
\begin{align}\notag
A_n^2&\sim \frac{\pi^2}{16}\left(1-4\left(1-\frac{1}{4p-2}\right)^2(2\beta-1)^2\right)\log n.
\end{align}
By an argument similar to the proof of Theorem \ref{mThm1}, we have
\begin{align}\notag
\sum_{k=1}^{[n^t]}E[X_{n,k}^2|\mathcal{F}_{k-1}^X]\sim\frac{A_{[n^t]}^2}{A_n^2}\sim t.
\end{align}
Then Condition (a$'$) is satisfied. 

Next we check Condition (b). Because $a_n\ge1$ and Eq. \eqref{Eq}, it follows that $|X_{n,k}|\le\frac{2}{A_n}$ a.s. Since $A_n\to\infty$ as $n\to\infty$, for any fixed $\varepsilon>0$, we have $P(|X_{n,k}|>\varepsilon)=0$ for large $n$. Then Condition (b) is satisfied. The rest of the proof is similar to that of Theorem \ref{mThm1}, and therefore we obtain the desired convergence \eqref{FCLT2}.
\end{proof}
\noindent
\section*{Acknowledgements} 
The author would like to thank Professor Kouji Yano, for his careful guidance and great support. The author thanks Ryoichiro Noda, for his insightful question and valuable advice, which initiated the author's research on a functional limit theorem for the DERW.
%%%%% references %%%%%
% 文献情報はMathSciNetから取得するのが便利
% http://www.ams.org/mathscinet/

\end{document}